\pdfoutput=1
\documentclass[12pt,reqno]{amsart}
\usepackage{latexsym}
\usepackage{amsmath}
\usepackage{amsthm}
\usepackage{amssymb}
\usepackage{graphicx}
\usepackage{amsfonts}
\usepackage{fullpage}
\usepackage{comment}
\usepackage{hyperref}

\newcommand{\R}{\mathbb{R}}

\newcommand{\der}{\mathrm{d}}

\newcommand{\sulut}[1]{\left( #1 \right)}

\DeclareMathOperator{\dive}{div}

\title{Inverse problems for $p$-Laplace type equations under monotonicity assumptions}

\author[Chang-Yu Guo]{Chang-Yu Guo}
\address{Department of Mathematics and Statistics, University of Jyv\"askyl\"a and Department of Mathematics, University of Fribourg}
\email{changyu.guo@unifr.ch}

\author[Manas Kar]{Manas Kar}
\address{Department of Mathematics and Statistics, University of Jyv\"askyl\"a}
\email{manas.m.kar@maths.jyu.fi}

\author[Mikko Salo]{Mikko Salo}
\address{Department of Mathematics and Statistics, University of Jyv\"askyl\"a}
\email{mikko.j.salo@jyu.fi}

\newcommand*{\eps}{\varepsilon}
\newcommand{\abs}[1]{\left| #1 \right|}

\newcommand{\norm}[1]{\left\| #1 \right\|}

\theoremstyle{plain}

\newtheorem{thm}{Theorem}[section]

\newtheorem{lemma}[thm]{Lemma}

\newtheorem*{question}{Question}

\theoremstyle{definition}

\newtheorem*{remark}{Remark}

\numberwithin{equation}{section}
\newcommand{\nocontentsline}[3]{}
\newcommand{\tocless}[2]{\bgroup\let\addcontentsline=\nocontentsline#1{#2}\egroup}
\renewcommand{\phi}{\varphi}

\newcommand{\mR}{\mathbb{R}}

\newcommand{\mdiv}{\mathrm{div}}

\def\vint_#1{\mathchoice%
          {\mathop{\kern 0.2em\vrule width 0.6em height 0.69678ex
depth -0.58065ex
                  \kern -0.8em \intop}\nolimits_{\kern -0.4em#1}}%
          {\mathop{\kern 0.1em\vrule width 0.5em height 0.69678ex
depth -0.60387ex
                  \kern -0.6em \intop}\nolimits_{#1}}%
          {\mathop{\kern 0.1em\vrule width 0.5em height 0.69678ex
              depth -0.60387ex
                  \kern -0.6em \intop}\nolimits_{#1}}%
          {\mathop{\kern 0.1em\vrule width 0.5em height 0.69678ex
depth -0.60387ex
                  \kern -0.6em \intop}\nolimits_{#1}}}
\def\vintslides_#1{\mathchoice%
          {\mathop{\kern 0.1em\vrule width 0.5em height 0.697ex depth -0.581ex
                  \kern -0.6em \intop}\nolimits_{\kern -0.4em#1}}%
          {\mathop{\kern 0.1em\vrule width 0.3em height 0.697ex depth -0.604ex
                  \kern -0.4em \intop}\nolimits_{#1}}%
          {\mathop{\kern 0.1em\vrule width 0.3em height 0.697ex depth -0.604ex
                  \kern -0.4em \intop}\nolimits_{#1}}%
          {\mathop{\kern 0.1em\vrule width 0.3em height 0.697ex depth -0.604ex
                  \kern -0.4em \intop}\nolimits_{#1}}}

\newcommand{\aveint}[2]{\mathchoice%
          {\mathop{\kern 0.2em\vrule width 0.6em height 0.69678ex
depth -0.58065ex
                  \kern -0.8em \intop}\nolimits_{\kern -0.45em#1}^{#2}}%
          {\mathop{\kern 0.1em\vrule width 0.5em height 0.69678ex
depth -0.60387ex
                  \kern -0.6em \intop}\nolimits_{#1}^{#2}}%
          {\mathop{\kern 0.1em\vrule width 0.5em height 0.69678ex
depth -0.60387ex
                  \kern -0.6em \intop}\nolimits_{#1}^{#2}}%
          {\mathop{\kern 0.1em\vrule width 0.5em height 0.69678ex
depth -0.60387ex
                  \kern -0.6em \intop}\nolimits_{#1}^{#2}}}

\numberwithin{equation}{section}

\begin{document}

\date{\today}

\dedicatory{Dedicated to Giovanni Alessandrini on the occasion of his 60th birthday}

\begin{abstract}
We consider inverse problems for $p$-Laplace type equations under monotonicity assumptions. In two dimensions, we show that any two conductivities satisfying $\sigma_1 \geq \sigma_2$ and having the same nonlinear Dirichlet-to-Neumann map must be identical. The proof is based on a monotonicity inequality and the unique continuation principle for $p$-Laplace type equations. In higher dimensions, where unique continuation is not known, we obtain a similar result for conductivities close to constant.
\end{abstract}

\maketitle

\section{Introduction} \label{sec_intro}
 
The inverse conductivity problem posed by Calder\'on asks if the electrical conductivity of a medium can be determined by voltage and current measurements on its boundary. If $\Omega \subset \mathbb{R}^n$ is a bounded open set representing the medium, and if $\sigma \in L^{\infty}_+(\Omega)$ is a function representing the electrical conductivity, then Ohm's and Kirchhoff's laws imply that given a boundary voltage $f$, the electrical potential $u$ in $\Omega$ will solve the conductivity equation 
\begin{align*}
\left\{ \begin{array}{rl} \text{div}(\sigma \nabla u) &\!\!\!= 0 \ \text{ in } \Omega, \\
u &\!\!\!= f \ \text{ on } \partial \Omega. \end{array} \right.
\end{align*}
Here and below we write 
\[
L^{\infty}_+(\Omega) = \{ \sigma \in L^{\infty}(\Omega) \,;\, \sigma \geq c_0 > 0 \text{ a.e.\ in $\Omega$ for some $c_0 > 0$} \}.
\]
If $X(\Omega)$ is a function space (such as the space $W^{1,\infty}(\Omega)$ of Lipschitz functions), we also write 
\begin{align*}
X_+(\Omega) = X(\Omega) \cap L^{\infty}_+(\Omega).
\end{align*}

The boundary measurements are encoded by the Dirichlet-to-Neumann map (DN map) 
\[
\Lambda_{\sigma}: f \mapsto \sigma \partial_{\nu} u|_{\partial \Omega}
\]
where $\sigma \partial_{\nu} u|_{\partial \Omega}$ is the electrical current flowing through the boundary, and the normal derivative $\partial_{\nu}$ is defined in a suitable weak sense. The inverse problem is to determine the conductivity $\sigma$ from knowledge of the DN map $\Lambda_{\sigma}$. This question, known as the Calder\'on problem, is a fundamental inverse problem with applications in industrial and medical imaging and having connections to many other inverse problems. Both the theoretical and applied aspects of the Calder\'on problem have been studied intensively in the last 35 years. We refer to the survey \cite{Uhlmann:2014} for more information.

In this paper we consider a nonlinear variant of the Calder\'on problem. Here the standard Ohm's law $j = -\sigma \nabla u$ relating the current $j$ and potential $u$ is replaced by the nonlinear law 
\[
j = -\sigma \abs{\nabla u}^{p-2} \nabla u
\]
where $p$ is a real number with $1 < p < \infty$. Combining this with Kirchhoff's law stating that $j$ is divergence free, we obtain the boundary value problem 
\begin{align*}
\left\{�\begin{array}{rl} \text{div}(\sigma \abs{\nabla u}^{p-2} \nabla u) &\!\!\!= 0 \ \text{ in } \Omega, \\
u &\!\!\!= f \ \text{ on } \partial \Omega. \end{array} \right.
\end{align*}
The boundary measurements are encoded by the nonlinear DN map 
\[
\Lambda_{\sigma}: f \mapsto \sigma \abs{\nabla u}^{p-2} \partial_{\nu} u|_{\partial \Omega}
\]
defined in a suitable weak sense. The inverse problem is to determine the conductivity $\sigma$ from knowledge of the nonlinear map $\Lambda_{\sigma}$.

The equation $\text{div}(\sigma \abs{\nabla u}^{p-2} \nabla u) = 0$ is called the weighted $p$-Laplace equation (with weight given by the positive function $\sigma$), and it is the Euler-Lagrange equation related to minimizing the $p$-Dirichlet energy $E(u) = \int_{\Omega} \sigma \abs{\nabla u}^p \,dx$. The case $p=2$ is just the linear conductivity equation, but if $p \neq 2$ this is a quasilinear degenerate elliptic equation. The $p$-Laplace equation appears as a model for nonlinear dielectrics, plastic moulding, electro-rheological and thermo-rheological fluids, fluids governed by a power law, viscous flows in glaciology, or plasticity. The limiting cases $p=0$ and $p=1$ also arise in hybrid imaging inverse problems such as ultrasound modulated electrical impedance tomography (UMEIT) and current density imaging (CDI). See the references in \cite{Brander:Kar:Salo:2014} for further information. The $p$-Laplace equation is of considerable mathematical interest as well, the case $p=n$ is useful in conformal geometry~\cite{Liimatainen:Salo:2012} and also the limiting cases $p=0,1,\infty$ are relevant. We refer to \cite{Heinonen:Kilpelainen:Martio:1993}, \cite{Lindqvist:2006}, \cite{Evans:2007} for further details on $p$-Laplace type equations.

The inverse problem of determining $\sigma$ from the nonlinear DN map $\Lambda_{\sigma}$ was introduced in \cite{Salo:Zhong:2012} as a nonlinear variant of the Calder\'on problem. There are several previous works related to Calder\'on type problems for nonlinear equations (see the references of \cite{Salo:Zhong:2012}), where the inverse problem is solved by linearizing the nonlinear DN map. However, the $p$-Laplace type model has the new feature that linearizations at constant boundary values do not give any new information, and thus genuinely nonlinear methods are required to treat the inverse problem. The work \cite{Salo:Zhong:2012} suggested a nonlinear version of the method of complex geometrical optics solutions that has been widely used in the original Calder\'on problem (see the survey \cite{Uhlmann:2014}). The nonlinear version of this method was based on $p$-harmonic functions introduced by Wolff \cite{Wolff:2007}.

We are aware of the following results on the inverse problem for $p$-Laplace type equations:
\begin{itemize}
\item 
(Boundary uniqueness \cite{Salo:Zhong:2012}) $\Lambda_{\sigma}$ determines $\sigma|_{\partial \Omega}$.
\item 
(Uniqueness for normal derivative \cite{Brander:2014}) $\Lambda_{\sigma}$ determines $\partial_{\nu} \sigma|_{\partial \Omega}$.
\item 
(Inclusion detection \cite{Brander:Kar:Salo:2014}) If $\sigma = 1$ in $\Omega \setminus \overline{D}$ and $\sigma \geq 1+\eps > 1$ in $D$ where $D \subset \Omega$ is an obstacle, then $\Lambda_{\sigma}$ determines the convex hull of $D$. Further results are given in \cite{Brander:Ilmavirta:Kar:2015}.
\end{itemize}
The first two results were based on Wolff type solutions and boundary determination arguments following Brown \cite{Brown:2001}. We remark that it would be interesting to see if also the boundary determination method based on singular solutions due to Alessandrini \cite{Alessandrini:1990} applies to $p$-Laplace type equations. The third result above extends the enclosure method for inclusion detection introduced by Ikehata \cite{Ikehata:2000b} to the $p$-Laplace case. The main new ingredient in \cite{Brander:Kar:Salo:2014} was a monotonicity inequality, which allows to estimate the difference of DN maps, $\Lambda_{\sigma_1} -\Lambda_{\sigma_2}$, under the condition $\sigma_1 \geq \sigma_2$. 

In this paper we continue the study of inverse problems for $p$-Laplace type equations. The main point is that a monotonicity assumption $\sigma_1 \geq \sigma_2$, together with the monotonicity inequality and the unique continuation principle for $p$-Laplace type equations in the plane \cite{Alessandrini:1987, bi84, Manfredi:1988, Alessandrini:Sigalotti:2001}, allows to establish interior uniqueness for the conductivities.

\begin{thm} \label{intro_thm1}
Let $\Omega \subset \mR^2$ be a bounded open set and let $1 < p < \infty$. If $\sigma_1, \sigma_2 \in W^{1,\infty}_+(\Omega)$ satisfy $\sigma_1 \geq \sigma_2$ \ in $\Omega$, then $\Lambda_{\sigma_1} = \Lambda_{\sigma_2}$ implies $\sigma_1 = \sigma_2$ in $\Omega$.
\end{thm}

In three and higher dimensions, the unique continuation principle even for the standard $p$-Laplace equation remains an important open question (see for instance \cite{Granlund:Marola:2014}). We obtain the following partial result under the additional assumption that one of the conductivities is close to a constant.

\begin{thm} \label{intro_thm2}
Let $\Omega \subset \mR^n$ with $n \geq 2$ be a bounded open set with $C^{1,\alpha}$ boundary where $0 < \alpha < 1$, let $1 < p < \infty$, and let $M > 0$. There is a constant $\eps = \eps(n,p,\alpha,\Omega,M) > 0$ such that if $\sigma_1, \sigma_2 \in C^{\alpha}(\overline{\Omega})$ satisfy $\sigma_1 \geq \sigma_2$ in $\Omega$ and 
\[
1/M \leq \sigma_1 \leq M, \qquad \norm{\sigma_1}_{C^{\alpha}(\overline{\Omega})} \leq M, \qquad \norm{\sigma_2 - 1}_{C^{\alpha}(\overline{\Omega})} \leq \eps,
\]
then $\Lambda_{\sigma_1} = \Lambda_{\sigma_2}$ implies $\sigma_1 = \sigma_2$ in $\Omega$.
\end{thm}

Both of the above theorems are based on the monotonicity inequality, Lemma \ref{needed_lemma}, and the existence of solutions whose gradient is nonvanishing in suitable sets. More precisely, Lemma \ref{needed_lemma} implies that for any $u \in W^{1,p}(\Omega)$ solving $\mdiv(\sigma_2 \abs{\nabla u}^{p-2} \nabla u) = 0$ in $\Omega$, one has 
\begin{align*}
(p-1) & \int_{\Omega} \frac{\sigma_2}{\sigma_1^{1/(p-1)}} \sulut{\sigma_1^{\frac{1}{p-1}} - \sigma_2^{\frac{1}{p-1}}} \abs{\nabla u}^{p} \,dx \leq \langle (\Lambda_{\sigma_1} - \Lambda_{\sigma_2})(u|_{\partial \Omega}), u|_{\partial \Omega} \rangle.
\end{align*}
Thus if $\sigma_1, \sigma_2 \in C_+(\overline{\Omega})$ satisfy $\sigma_1 \geq \sigma_2$ and $\Lambda_{\sigma_1} = \Lambda_{\sigma_2}$, it follows that 
\[
\abs{\nabla u}^p = 0 \text{ a.e.\ in $E$}
\]
for any solution $u$, where $E = \{Êx \in \Omega \,;\, \sigma_1(x) > \sigma_2(x) \}$. We would like to show that $\sigma_1 = \sigma_2$, or that $E$ is empty. But if $E$ would be nonempty, then all solutions $u$ would satisfy $\nabla u = 0$ in the open set $E$. It is thus enough to exhibit one solution $u$ with $\nabla u \neq 0$ somewhere in $E$.

If $\sigma \in C^{\alpha}_+(\overline{\Omega})$ for some $\alpha > 0$, we define the set of weak solutions 
\[
S_{\sigma} = \{ u \in W^{1,p}(\Omega) \,;\, \text{$\mathrm{div}(\sigma \abs{\nabla u}^{p-2} \nabla u) = 0$ in $\Omega$} \}.
\]
Each $u \in S_{\sigma}$ is locally $C^1$ (see e.g.\ \cite{Lieberman:1988}), and we let $\mathcal{C}(u)$ be the critical set of $u$, 
\[
\mathcal{C}(u) = \{ x \in \Omega \,;\, \abs{\nabla u}(x) = 0 \}.
\]
The study of critical sets of solutions is of independent interest, and there are a number of results in the case $p=2$ and also in the two-dimensional case when $p \neq 2$ (see \cite{Alessandrini:Sigalotti:2001, Cheeger:Naber:Valtorta:2015} and references therein). The following question is relevant in our context, and further answers to this question would imply improvements in the above theorems when $n \geq 3$:

\begin{question}
Let $\Omega \subset \mR^n$ be a bounded connected open set and let $1 < p < \infty$. Given $\sigma \in C^{\alpha}_+(\overline{\Omega})$, consider the following statements:
\begin{enumerate}
\item[(a)]
There is $u \in S_{\sigma}$ such that $\mathcal{C}(u)$ has Lebesgue measure zero.
\item[(b)]
For any set $E \subset \Omega$ of positive Lebesgue measure there is $u \in S_{\sigma}$ such that $\nabla u|_E$ is not zero a.e.\ in $E$.
\item[(c)]
For any open set $U \subset \Omega$ there is $u \in S_{\sigma}$ such that $\nabla u|_U$ is not zero a.e.\ in $U$.
\end{enumerate}
For which $\sigma \in C^{\alpha}_+(\overline{\Omega})$ does (a), (b), or (c) hold?
\end{question}

Clearly (a) $\implies$ (b) $\implies$ (c). We note that (a) holds for constant conductivities, or for conductivities only depending on $n-1$ variables (in these cases there is a linear function which is a solution with nonvanishing gradient). Also, (a) holds in two dimensions at least for Lipschitz $\sigma$, since $\mathcal{C}(u)$ for nonconstant $u \in S_{\sigma}$ is the set of zeros of a quasiregular map and hence has measure zero (see \cite{Alessandrini:Sigalotti:2001} or Appendix \ref{sec_appendix}). Finally, the weak unique continuation principle would imply (c) since then $\mathcal{C}(u)$ has empty interior for any nonconstant $u \in S_{\sigma}$.

We remark that in the linear case $p=2$, uniqueness results for the inverse problem even without monotonicity assumptions have been known for a long time (see the survey \cite{Uhlmann:2014}). Monotonicity arguments go back to \cite{Alessandrini1989, KangSeoSheen1997, AlessandriniRosset1998, Ikehata1998, AlessandriniRossetSeo2000, Ide:Isozaki:Nakata:Siltanen:Uhlmann:2007}, and recently they have been combined with the method of localized potentials introduced in \cite{Gebauer:2008} to obtain reconstruction algorithms \cite{Harrach:2012, Harrach:Ullrich:2013}. However, the unique continuation principle and the Runge approximation property play a role in these arguments, and these facts are not known for $p$-Laplace type equations in dimensions $n \geq 3$.

This paper is organized as follows. Section \ref{sec_intro} is the introduction. In Section \ref{sec_twod} we establish the monotonicity inequality and the two-dimensional result, Theorem \ref{intro_thm1}. Section \ref{INTERIOR} proves Theorem \ref{intro_thm2} which is valid in any dimension, by a perturbation argument around constant conductivities. We will do the proofs for the slightly more general equation 
\[
\text{div}(\sigma \abs{A \nabla u \cdot \nabla u}^{(p-2)/2} A \nabla u) = 0 \ \ \text{ in } \Omega
\]
where $\sigma$ is a positive scalar function and $A$ is a positive definite matrix function. Finally, Appendix \ref{sec_appendix} contains some interpolation and unique continuation results required in the proofs.

\subsection*{Acknowledgements}

C.Y.\ Guo was supported by the Magnus Ehrnrooth foundation. M.\ Kar\ and M.\ Salo\ were partly supported by an ERC Starting Grant (grant agreement no 307023) and by the Academy of Finland through the Centre of Excellence in Inverse Problems Research. M.\ Salo\ was also supported by CNRS. C.Y.\ Guo would like to thank the ICMAT program on Analysis and Geometry on Metric Spaces in 2015, and M.\ Kar\ and M.\ Salo\  would like to thank the Institut Henri Poincar\'e Program on Inverse Problems in 2015, where part of this work was carried out. M.\ Salo\ would like to thank Niko Marola for helpful discussions.

\section{Interior uniqueness in the plane} \label{sec_twod}

Given a bounded open set $\Omega\subset\mathbb{R}^2$ and a conductivity $\sigma\in L_{+}^{\infty}(\Omega)$, we consider the Dirichlet problem for the following $p$-Laplace type equation where $1<p<\infty$, 
\begin{equation}  \label{p-lapPN0}
\begin{cases}
\text{div}(\sigma\abs{A\nabla u\cdot \nabla u}^{(p-2)/2} A\nabla u) = 0 \ \text{in}\ \Omega, \\
u = f \ \text{on}\ \partial\Omega,
\end{cases}
\end{equation}
where $A \in L^{\infty}_+(\Omega, \mR^{n \times n}) $, meaning that $A = (a_{jk})$ where $a_{jk} \in L^{\infty}(\Omega)$, $a_{jk} = a_{kj}$, and for some $c_0 > 0$ one has $\sum_{j,k=1}^n a_{jk}(x) \xi_j \xi_k \geq c_0 \abs{\xi}^2$ for a.e.\ $x \in \Omega$ and for all $\xi \in \mR^n$.

The problem \eqref{p-lapPN0} is well posed in $W^{1,p}(\Omega)$ for a given Dirichlet boundary data $f\in W^{1,p}(\Omega)$ (the boundary values are understood so that $u-f \in W_0^{1,p}(\Omega)$), see for instance~\cite{D'Onofrio:Iwaniec:2005,Heinonen:Kilpelainen:Martio:1993,Lindqvist:2006,Salo:Zhong:2012}. The solution $u$ minimizes the $p$-Dirichlet energy 
\[
E_p(v) = \int_{\Omega}\sigma|A\nabla v\cdot \nabla v|^{p/2}dx 
\]
over all $v\in W^{1,p}(\Omega)$ with $v-f \in W^{1,p}_0(\Omega)$.

We formally define the nonlinear DN map by
\[
\Lambda_{\sigma}: f \mapsto \sigma |A\nabla u\cdot \nabla u|^{(p-2)/2}A\nabla u\cdot \nu|_{\partial\Omega},
\]
where $u\in W^{1,p}(\Omega)$ satisfies \eqref{p-lapPN0}. More precisely, $\Lambda_{\sigma}$ is a nonlinear map $X \rightarrow X'$ where $X$ is the abstract trace space $X = W^{1,p}(\Omega)/W_0^{1,p}(\Omega)$ and $X'$ denotes the dual of $X$, and $\Lambda_{\sigma}$ is defined by the relation 
\begin{equation} \label{dnmap_weak_definition}
\langle\Lambda_{\sigma}(f), g\rangle = \int_{\Omega} \sigma \abs{A\nabla u\cdot \nabla u}^{(p-2)/2} A\nabla u \cdot \nabla v \,dx, \qquad f, g \in X,
\end{equation}
where $u \in W^{1,p}(\Omega)$ is the unique solution of $\text{div}(\sigma \abs{A\nabla u\cdot \nabla u}^{(p-2)/2} A\nabla u) = 0$ in $\Omega$ with $u|_{\partial \Omega} = f$, and $v$ is any function in $W^{1,p}(\Omega)$ with $v|_{\partial \Omega} = g$. Here $\langle\,\cdot\, , \,\cdot\, \rangle$ is the duality between $X'$ and $X$. If $\partial \Omega$ has Lipschitz boundary, the trace space $X$ can be identified with the Besov space $B^{1-1/p}_{pp}(\partial \Omega)$. Physically $\Lambda_\sigma (f)$ is the current flux density caused by the boundary potential~$f$. See \cite[Appendix]{Salo:Zhong:2012} and \cite{Hauer:2014} for further properties of the DN map.

The following is the main result of this section.

\begin{thm}\label{main_theorem}
Let $\Omega \subset \mR^2$ be a bounded open set, let $A \in W^{1,\infty}(\Omega, \mR^{n \times n})$ be a symmetric positive definite matrix function, and let $\sigma_1, \sigma_2 \in W^{1,\infty}_+(\Omega)$ be two conductivities such that $\sigma_1\geq \sigma_2$ in $\Omega$. If $\Lambda_{\sigma_1} = \Lambda_{\sigma_2}$, then $\sigma_1 = \sigma_2$ in $\Omega$.
\end{thm}

The proof is based on a monotonicity inequality and the unique continuation principle for solutions of \eqref{p-lapPN0}. Let us first consider the monotonicity inequality, which holds true in any dimension $n\geq 2$. In the linear case $p=2$, the following inequality is well known (see references in the introduction). For $p \neq 2$ this inequality was proved in \cite{{Brander:Kar:Salo:2014}} in the case $A = I$. The proof for general $A$ is almost identical, but we give it here for completeness. 

\begin{lemma}\label{needed_lemma}
Let $\Omega \subset \mR^n$ be a bounded open set where $n \geq 2$, let $\sigma_1, \sigma_2 \in L^{\infty}_+(\Omega)$, let $A \in C(\overline{\Omega}, \mR^{n \times n})$ be a symmetric positive definite matrix function, and let $1 < p < \infty$. If $f \in W^{1,p}(\Omega)$, then 
\begin{align*}
(p-1) & \int_{\Omega} \frac{\sigma_2}{\sigma_1^{1/(p-1)}} \sulut{\sigma_1^{\frac{1}{p-1}} - \sigma_2^{\frac{1}{p-1}}} \abs{A\nabla u_2\cdot \nabla u_2}^{p/2} \,dx \\
& \leq \langle(\Lambda_{\sigma_1} - \Lambda_{\sigma_2})f, f\rangle 
 \leq \int_{\Omega} (\sigma_1 - \sigma_2) \abs{A\nabla u_2\cdot \nabla u_2}^{p/2} \,dx,
\end{align*}
where $u_2 \in W^{1,p}(\Omega)$ solves $\mdiv(\sigma_2 \abs{A\nabla u_2\cdot \nabla u_2}^{(p-2)/2} A\nabla u_2) = 0$ in $\Omega$ with $u_2|_{\partial \Omega} = f$.
\end{lemma}
We emphasize that if $\sigma_1 \geq \sigma_2$, then all the terms in the inequality are nonnegative, while if $\sigma_1 \leq \sigma_2$, then they are nonpositive.

\begin{proof}
Let $u_1, u_2 \in W^{1,p}(\Omega)$ be the solutions of the Dirichlet problem for the $p$-Laplace type equation,
\begin{equation}  \label{p-lapp}
	\begin{cases}
		\mdiv(\sigma \abs{A\nabla u\cdot \nabla u}^{(p-2)/2} A\nabla u) = 0 \ \text{in}\ \Omega, \\
		u = f \ \text{on}\ \partial\Omega,
	\end{cases}
\end{equation} 
corresponding to the conductivities $\sigma = \sigma_1$ and $\sigma = \sigma_2$ respectively.

Note that the solution of \eqref{p-lapp} can be characterized as the unique minimizer of the energy functional
\[
E_p(v) = \int_{\Omega}\sigma|A\nabla v\cdot \nabla v|^{p/2}dx
\]
over the set $\{v\in W^{1,p}(\Omega) ; v-f \in W_{0}^{1,p}(\Omega)\}$ (see \cite[Appendix]{Salo:Zhong:2012}).
Therefore, we obtain the following one sided inequality for the difference of DN maps:
\begin{align*}
	\langle(\Lambda_{\sigma_1} - \Lambda_{\sigma_2})f, f\rangle
	& =\int_{\Omega}\sigma_1|A\nabla u_1\cdot \nabla u_1|^{p/2}dx - \int_{\Omega}\sigma_2|A\nabla u_2\cdot \nabla u_2|^{p/2}dx \\
	& \leq \int_{\Omega}(\sigma_1-\sigma_2)|A\nabla u_2\cdot \nabla u_2|^{p/2}dx.
\end{align*}

Since $A$ is symmetric positive definite, $A=B^{\top}B$ for some symmetric matrix function $B \in C(\overline{\Omega}, \mR^{n \times n})$. The existence of such matrix $B$ is due to the Lemma \ref{existence_B} in the Appendix. To obtain the other side of the inequality, note that 
\[
A\nabla u_1\cdot \nabla u_2=B\nabla u_1\cdot B\nabla u_2.
\]
Let $\beta > 0$ be a real number (whose value will be fixed later). Using \eqref{dnmap_weak_definition} several times together with the fact that $u_1|_{\partial \Omega} = u_2|_{\partial \Omega}$, we may rewrite the difference of DN maps as follows:
\begin{align*}
	\langle(\Lambda_{\sigma_1} - \Lambda_{\sigma_2})f, f\rangle
	&= \int_{\Omega}\beta\sigma_2|A\nabla u_2\cdot \nabla u_2|^{p/2}  \\
	&\qquad - \sulut{(1 + \beta)\sigma_2\abs{A\nabla u_2\cdot \nabla u_2}^{\frac{p-2}{2}}A\nabla u_2 \cdot \nabla u_2  - \sigma_1 |A\nabla u_1\cdot \nabla u_1|^{p/2}}  \der x \\
	&= \int_{\Omega}\beta\sigma_2|B\nabla u_2|^p - \left((1+\beta)\sigma_2|B\nabla u_2|^{p-2}B\nabla u_2\cdot B\nabla u_1 - \sigma_1|B\nabla u_1|^p\right)\der x.
\end{align*}
Now, by applying Young's inequality $\abs{ab} \leq \frac{\abs{a}^p}{p} + \frac{\abs{b}^{p'}}{p'}$ where $1/p+1/{p'} = 1$, we have 
\begin{align*}
	&(1+\beta)\sigma_2|B\nabla u_2|^{p-2}B\nabla u_2\cdot B\nabla u_1 - \sigma_1|B\nabla u_1|^p \\
	& = \frac{1+\beta}{p^{1/p}}\frac{\sigma_2}{\sigma_{1}^{1/p}}|B\nabla u_2|^{p-2}B\nabla u_2 \cdot  p^{1/p}\sigma_{1}^{1/p}B\nabla u_1 - \sigma_1|B\nabla u_1|^p \\
	& \leq \frac{1}{p'}\left(\frac{1+\beta}{p^{1/p}}\right)^{p'}\frac{\sigma_{2}^{p'}}{\sigma_{1}^{1/(p-1)}}|B\nabla u_2|^p + \sigma_1|B\nabla u_1|^p - \sigma_1|B\nabla u_1|^p \\
	& = \frac{1}{p'}\left(1+\beta\right)^{p'}\frac{1}{p^{1/(p-1)}}\frac{\sigma_{2}^{p'}}{\sigma_{1}^{1/(p-1)}}|B\nabla u_2|^p.
\end{align*}

Therefore
\begin{align}
	 \langle(\Lambda_{\sigma_1} - \Lambda_{\sigma_2})f, f\rangle
	& \geq \int_{\Omega}\left(\beta\sigma_2 - \frac{1}{p'}\left(1+\beta\right)^{p'}\frac{1}{p^{1/(p-1)}}\frac{\sigma_{2}^{p'}}{\sigma_{1}^{1/(p-1)}}\right)|B\nabla u_2|^pdx \nonumber \\
	& = \int_{\Omega}\frac{\beta\sigma_2}{\sigma_{1}^{1/(p-1)}}\left(\sigma_{1}^{\frac{1}{p-1}} - \frac{1}{p'}\frac{(1+\beta)^{p'}}{\beta}\left(\frac{1}{p}\right)^{\frac{1}{p-1}}\sigma_{2}^{\frac{1}{p-1}}\right)|B\nabla u_2|^pdx. \label{DNkey}
\end{align}

Note that $\frac{(1+\beta)^{p'}}{\beta} \rightarrow \infty$ as $\beta\rightarrow \infty$ or $\beta\rightarrow 0$. 
So, the function $\beta \rightarrow \frac{(1+\beta)^{p'}}{\beta}$ attains its minimum at $\beta = p-1$.
Thus, we choose $\beta = p-1$ so that from \eqref{DNkey}, we obtain the required inequality. 
\end{proof}

Next we consider the unique continuation principle for solutions of the $p$-Laplace type equation
\begin{align}\label{eq:UCP equation}
\mdiv(\sigma\abs{A\nabla u\cdot \nabla u}^{(p-2)/2} A\nabla u) = 0.
\end{align}
The case when $\sigma$ is constant and $A=I$ is well-known due to the work of Alessandrini \cite{Alessandrini:1987}, Bojarski-Iwaniec \cite{bi84} and Manfredi \cite{Manfredi:1988}, namely, if $u$ is a solution of the $p$-Laplace equation
\begin{align*}
\mdiv(\abs{\nabla u\cdot \nabla u}^{(p-2)/2} \nabla u) = 0
\end{align*}
in a planar domain $\Omega\subset \R^2$ and if $u$ is constant in an open subset of $\Omega$, then it is actually constant in the whole domain $\Omega$. The proof of Alessandrini involves a linear equation for $\log\,\abs{\nabla u}$, whereas the proof of Bojarski-Iwaniec (see also \cite[Chapter 16]{aim09} for a presentation) uses that complex gradients of solutions of the $p$-Laplace equation are quasiregular mappings, and that non-constant quasiregular mappings are discrete and open.

The unique continuation principle holds for solutions of~\eqref{eq:UCP equation} as well at least when the coefficients are Lipschitz, see \cite[Proposition 3.3]{Alessandrini:Sigalotti:2001}.

\begin{thm}\label{thm:UCP 2d}
If $\Omega$ is a domain in $\mathbb{R}^2$, $A \in W^{1,\infty}(\Omega, \mR^{n \times n})$ is a symmetric positive definite matrix function and $\sigma \in W^{1,\infty}_+(\Omega)$, and if $u \in W^{1,p}_{\mathrm{loc}}(\Omega)$ is a solution of~\eqref{eq:UCP equation} which is constant in an open nonempty subset of $\Omega$, then $u$ is identically constant in $\Omega$.
\end{thm}

In the appendix, for possible later purposes we sketch an alternative proof of Theorem~\ref{thm:UCP 2d} for $A=I$ and $\sigma$ Lipschitz continuous, based on the theory of Beltrami equations, following the approach introduced by Bojarski and Iwaniec~\cite{bi84}.

\begin{proof}[Proof of Theorem \ref{main_theorem}]
We argue by contradiction and suppose that $\sigma_1(x_0) > \sigma_2(x_0)$ for some $x_0 \in \overline{\Omega}$. Since $\sigma_1$ and $\sigma_2$ are continuous, there exists some open ball $D \subset \Omega$ so that $\sigma_1 - \sigma_2 > 0$ in $D$.

Let $u_2\in W^{1,p}(\Omega)$ be a solution of $\mdiv(\sigma_2 \abs{A\nabla u\cdot \nabla u}^{(p-2)/2} A\nabla u) = 0$ in $\Omega$, with non-constant Dirichlet data $f \in W^{1,p}(\Omega)$ (i.e.\ $f-C \notin W^{1,p}_0(\Omega)$ for any constant $C$).
Using the left hand side of the monotonicity inequality (Lemma \ref{needed_lemma}), the assumptions that $\sigma_1 \geq \sigma_2$ and $\Lambda_{\sigma_1} = \Lambda_{\sigma_2}$, and the fact that $\sigma_{1}^{\frac{1}{p-1}} - \sigma_{2}^{\frac{1}{p-1}} \geq c_0 > 0$ in $D$, we deduce that
\begin{equation}\label{eq2}
 \int_{D}|A\nabla u_2\cdot \nabla u_2|^{p/2} dx \leq 0.
 \end{equation}
Then $|A\nabla u_2\cdot \nabla u_2|^{p/2} = 0$ a.e.\ in $D$, and the uniform ellipticity condition for $A$ implies that $\nabla u_2 = 0$ a.e.\ in $D$, i.e., $u_2$ is constant on $D$. By the unique continuation principle (Theorem \ref{thm:UCP 2d}), we know that $u_2$ is constant on $\Omega$. This contradicts the fact that $u_2$ had non-constant Dirichlet data $f$.
\end{proof}

\remark
Theorem~\ref{main_theorem} would remain valid in higher dimensions if the unique continuation principle would hold for solutions of~\eqref{eq:UCP equation}.

\section{Interior uniqueness in higher dimensions}\label{INTERIOR}

In this section, we will consider interior uniqueness for $p$-Laplace type inverse problems in dimensions $n\geq 3$ (the method also works when $n=2$). We will show that two conductivities $\sigma_1, \sigma_2$ that satisfy $\sigma_1 \geq \sigma_2$ in $\Omega$ and $\Lambda_{\sigma_1} = \Lambda_{\sigma_2}$ must be identical in $\Omega$, under the additional assumption that one of the conductivities (as well as the matrix $A$) is close to constant.

Our main result reads as follows:

\begin{thm}\label{thm2}
Let $\Omega \subset \mR^n$, $n \geq 2$, be a bounded open set with $C^{1,\alpha}$ boundary where $0 < \alpha < 1$, let $1 < p < \infty$, and let $M > 0$. There exists $\eps = \eps(n,p,\alpha,\Omega,M) > 0$ such that for any $\sigma_1, \sigma_2 \in C^{\alpha}(\overline{\Omega})$ and for any symmetric positive definite $A \in C^{\alpha}(\overline{\Omega}, \mR^{n \times n})$ satisfying 
\begin{gather*}
1/M \leq \sigma_1 \leq M \ \text{ in $\Omega$}, \\
\norm{\sigma_j}_{C^{\alpha}(\overline{\Omega})} + \norm{A}_{C^{\alpha}(\overline{\Omega})} \leq M, \\
\norm{\sigma_2-1}_{L^{\infty}(\Omega)} + \norm{A-I}_{L^{\infty}(\Omega)} \leq \eps,
\end{gather*}
the conditions $\sigma_1 \geq \sigma_2$ in $\Omega$ and $\Lambda_{\sigma_1} = \Lambda_{\sigma_2}$ imply that 
\[
\sigma_1 = \sigma_2 \ \text{in}\ \Omega.
\]
\end{thm}

The proof is again based on the monotonicity inequality and the fact that one can find solutions whose critical sets are small (in fact empty). However, since it is not known if the unique continuation principle holds for our equations in dimensions $n \geq 3$, we will construct solutions with nonvanishing gradient by perturbing a linear function $u_0(x) = x_1$ which solves the constant coefficient $p$-Laplace equation 
\[
\text{div}(|\nabla u_0|^{p-2}\nabla u_0) = 0 \ \text{in}\ \mathbb{R}^n.
\]
Alternatively, one could also perturb the complex geometrical optics or Wolff type solutions of the $p$-Laplace equation considered in \cite{Salo:Zhong:2012} which also have nonvanishing gradient.
 
The first step is to show that if $u_0$ solves $\mdiv(\sigma_0 \abs{A_0 \nabla u_0 \cdot \nabla u_0}^{\frac{p-2}{2}} \nabla u_0) = 0$ in $\Omega$, and if one perturbs $\sigma_0$ and $A_0$ slightly, then the solution $u_1$ of the perturbed equation will stay close to $u_0$ in $W^{1,p}$ norm if $u_1|_{\partial \Omega} = u_0|_{\partial \Omega}$.
 
\begin{lemma}\label{higher_dim_lemma1}
Let $\Omega \subset \mR^n$, $n \geq 2$, be a bounded open set, let $1 < p < \infty$, and let $M > 0$. There is $C = C(n,p,\Omega,M)$ such that for any $\sigma_0, \sigma_1 \in L^{\infty}_+(\Omega)$ and $A_0, A_1 \in L^{\infty}_+(\Omega, \mR^{n \times n})$ satisfying 
\[
1/M \leq \sigma_j \leq M, \qquad 1/M \leq A_j \leq M \qquad \text{a.e.\ in $\Omega$},
\]
one has 
\[
\norm{\nabla u_1 - \nabla u_0}_{L^p(\Omega)} \leq C (\norm{\sigma_1-\sigma_0}_{L^{\infty}(\Omega)} + \norm{A_1-A_0}_{L^{\infty}(\Omega)})^{\min\{ \frac{1}{p-1}, 1 \}} \norm{\nabla u_0}_{L^p(\Omega)}
\]
whenever $u_0, u_1 \in W^{1,p}(\Omega)$ solve 
\[
\mdiv(\sigma_j \abs{A_j \nabla u_j \cdot \nabla u_j}^{\frac{p-2}{2}} A_j \nabla u_j) = 0 \quad \text{in $\Omega$}
\]
and satisfy $u_1-u_0 \in W^{1,p}_0(\Omega)$.
\end{lemma}
\begin{proof}
Consider the expression 
\[
I := \int_{\Omega} (\abs{\nabla u_1} + \abs{\nabla u_0})^{p-2} \abs{\nabla u_1 - \nabla u_0}^2 \,dx.
\]
We will prove the estimate  
\begin{equation} \label{i_estimate}
I \leq C (\norm{\sigma_1-\sigma_0}_{L^{\infty}} + \norm{A_1-A_0}_{L^{\infty}}) \norm{\nabla u_0}_{L^p}^{p-1} \norm{\nabla u_1-\nabla u_0}_{L^p}.
\end{equation}
This implies the statement in the lemma: if $p \geq 2$ the triangle inequality gives 
\[
\int_{\Omega} \abs{\nabla u_1-\nabla u_0}^p \,dx \leq  \int_{\Omega} (\abs{\nabla u_1} + \abs{\nabla u_0})^{p-2} \abs{\nabla u_1 - \nabla u_0}^2 \,dx = I
\]
and \eqref{i_estimate} yields 
\begin{equation} \label{pgeqtwo_estimate}
\norm{\nabla u_1 - \nabla u_0}_{L^p} \leq C (\norm{\sigma_1-\sigma_0}_{L^{\infty}} + \norm{A_1-A_0}_{L^{\infty}})^{\frac{1}{p-1}} \norm{\nabla u_0}_{L^p}.
\end{equation}
On the other hand, if $1 < p < 2$ we write 
\[
\int_{\Omega} \abs{\nabla u_1-\nabla u_0}^p \,dx = \int_{\Omega} \left[ (\abs{\nabla u_1} + \abs{\nabla u_0})^{p-2} \abs{\nabla u_1 - \nabla u_0}^2 \right]^{p/2} (\abs{\nabla u_1} + \abs{\nabla u_0})^{\frac{p(2-p)}{2}} \,dx
\]
and use H\"older's inequality with exponents $q = 2/p$ and $q' = 2/(2-p)$ to get 
\[
\norm{\nabla u_1 - \nabla u_0}_{L^p}^p \leq I^{p/2} \left( \int_{\Omega} (\abs{\nabla u_1} + \abs{\nabla u_0})^p \,dx \right)^{\frac{2-p}{2}} \leq C I^{p/2} (\norm{\nabla u_1}_{L^p}^p + \norm{\nabla u_0}_{L^p}^p)^{\frac{2-p}{2}}.
\]
One also has $\norm{\nabla u_1}_{L^p} \leq C \norm{\nabla u_0}_{L^p}$ (this can be seen by integrating the equation for $u_1$ against the test function $u_1-u_0 \in W^{1,p}_0(\Omega)$). Using \eqref{i_estimate} yields 
\begin{equation} \label{plesstwo_estimate}
\norm{\nabla u_1 - \nabla u_0}_{L^p} \leq C (\norm{\sigma_1-\sigma_0}_{L^{\infty}} + \norm{A_1-A_0}_{L^{\infty}}) \norm{\nabla u_0}_{L^p}.
\end{equation}
The lemma follows by combining \eqref{pgeqtwo_estimate} (when $p \geq 2$) and \eqref{plesstwo_estimate} (when $1 < p < 2$).

It remains to show \eqref{i_estimate}. For any fixed $x$ (outside a set of measure zero), we may factorize $A_j = B_j^t B_j$ so that one has $\abs{B_j \xi}^2 = A_j \xi \cdot \xi$ and 
\[
\frac{1}{M} \abs{\xi}^2 \leq \abs{B_j \xi}^2 \leq M \abs{\xi}^2, \qquad \xi \in \mR^n.
\]
Using a basic inequality (see e.g.\ \cite[equation (A.4)]{Salo:Zhong:2012}) we have, for some $C = C(n,p,M)$ which may change from line to line and for a.e.\ $x$, 
\begin{align*}
 &(\abs{\nabla u_1} + \abs{\nabla u_0})^{p-2} \abs{\nabla u_1 - \nabla u_0}^2 \leq C (\abs{B_1 \nabla u_1} + \abs{B_1 \nabla u_0})^{p-2} \abs{B_1 \nabla u_1 - B_1 \nabla u_0}^2 \\
 &\qquad \leq C (\abs{B_1 \nabla u_1}^{p-2} B_1 \nabla u_1 - \abs{B_1 \nabla u_0}^{p-2} B_1 \nabla u_0) \cdot (B_1 \nabla u_1 - B_1 \nabla u_0) \\
 &\qquad \leq C \sigma_1 (\abs{A_1 \nabla u_1 \cdot \nabla u_1}^{\frac{p-2}{2}} A_1 \nabla u_1 - \abs{A_1 \nabla u_0 \cdot \nabla u_0}^{\frac{p-2}{2}} A_1 \nabla u_0) \cdot (\nabla u_1 - \nabla u_0).
\end{align*}
Using that $u_j$ are solutions and $u_1-u_0 \in W^{1,p}_0(\Omega)$, it follows that 
\begin{align*}
I &\leq C \int_{\Omega} \sigma_1 (\abs{A_1 \nabla u_1 \cdot \nabla u_1}^{\frac{p-2}{2}} A_1 \nabla u_1 - \abs{A_1 \nabla u_0 \cdot \nabla u_0}^{\frac{p-2}{2}} A_1 \nabla u_0) \cdot (\nabla u_1 - \nabla u_0) \,dx \\
 &=  -C \int_{\Omega} \sigma_1 \abs{A_1 \nabla u_0 \cdot \nabla u_0}^{\frac{p-2}{2}} A_1 \nabla u_0 \cdot (\nabla u_1 - \nabla u_0) \,dx \\
 &=  -C \int_{\Omega} (\sigma_1 \abs{A_1 \nabla u_0 \cdot \nabla u_0}^{\frac{p-2}{2}} A_1 \nabla u_0 - \sigma_0 \abs{A_0 \nabla u_0 \cdot \nabla u_0}^{\frac{p-2}{2}} A_0 \nabla u_0) \cdot (\nabla u_1 - \nabla u_0) \,dx
\end{align*}
Writing $\sigma_1 = (\sigma_1-\sigma_0) + \sigma_0$ and using the H\"older inequality, we get 
\begin{multline*}
I \leq C \norm{\sigma_1-\sigma_0}_{L^{\infty}} \norm{\nabla u_0}_{L^p}^{p-1} \norm{\nabla u_1-\nabla u_0}_{L^p} \\
 + C \norm{\abs{A_1 \nabla u_0 \cdot \nabla u_0}^{\frac{p-2}{2}} A_1 \nabla u_0 - \abs{A_0 \nabla u_0 \cdot \nabla u_0}^{\frac{p-2}{2}} A_0 \nabla u_0}_{L^{p/(p-1)}} \norm{\nabla u_1-\nabla u_0}_{L^p}.
\end{multline*}
Writing $A_1 \nabla u_0 = (A_1 \nabla u_0 - A_0 \nabla u_0) + A_0 \nabla u_0$ and using that $\abs{a_1^{\frac{p-2}{2}}-a_0^{\frac{p-2}{2}}} \leq C \abs{a_1-a_0}$ when $1/M \leq a_j \leq M$ (choosing $a_j = A_j \frac{\nabla u_0}{\abs{\nabla u_0}} \cdot \frac{\nabla u_0}{\abs{\nabla u_0}}$), we obtain \eqref{i_estimate}.
\end{proof}

We now interpolate the $W^{1,p}$ control of $u_1-u_0$ in Lemma \ref{higher_dim_lemma1} with the uniform bounds obtained from the $C^{1,\beta}$ regularity theory of $p$-Laplace type equations to show that $\nabla u_1$ is actually close to $\nabla u_0$ in $L^{\infty}$.

\begin{lemma}\label{higher_dim_lemma2}
Let $\Omega \subset \mR^n$, $n \geq 2$, be a bounded open set with $C^{1,\alpha}$ boundary where $0 < \alpha < 1$, let $1 < p < \infty$, and let $M > 0$. There exist $C > 0$ and $\gamma > 0$, depending on $n,p,\alpha,\Omega,M$, such that for any $\sigma_0, \sigma_1 \in C^{\alpha}(\overline{\Omega})$ and $A_0, A_1 \in L^{\infty}_+(\Omega, \mR^{n \times n})$ satisfying 
\begin{gather*}
1/M \leq \sigma_j, A_j \leq M \ \text{ in $\Omega$}, \\
\norm{\sigma_j}_{C^{\alpha}(\overline{\Omega})} + \norm{A_j}_{C^{\alpha}(\overline{\Omega})} \leq M,
\end{gather*}
and for any $f \in C^{1,\alpha}(\overline{\Omega})$ satisfying 
\[
\norm{f}_{C^{1,\alpha}(\overline{\Omega})} \leq M,
\]
one has 
\[
\norm{\nabla u_1 - \nabla u_0}_{L^{\infty}(\Omega)} \leq C (\norm{\sigma_1-\sigma_0}_{L^{\infty}(\Omega)} + \norm{A_1-A_0}_{L^{\infty}(\Omega)})^{\gamma}
\]
whenever $u_0, u_1 \in W^{1,p}(\Omega)$ solve 
\[
\mdiv(\sigma_j \abs{A_j \nabla u_j \cdot \nabla u_j}^{\frac{p-2}{2}} A_j \nabla u_j) = 0 \quad \text{in $\Omega$}
\]
and satisfy $u_1|_{\partial \Omega} = u_0|_{\partial \Omega} = f|_{\partial \Omega}$.
\end{lemma}
\begin{proof}
Under the stated assumptions, the weak solutions $u_0$ and $u_1$ are $C^{1,\beta}$ regular up to the boundary, see for instance \cite{Lieberman:1988}. More precisely, there exists $\beta = \beta(n,p,\alpha,\Omega,M)$ with $0 < \beta < 1$ so that $u_0$ and $u_1$ satisfy 
\begin{equation} \label{uj_conebeta_regularity}
\norm{u_j}_{C^{1,\beta}(\overline{\Omega})} \leq C
\end{equation}
where $C = C(n,p,\alpha,\Omega,M)$ may change from line to line. Clearly also $\norm{u_j}_{W^{1,p}(\Omega)} \leq C$. It follows from Lemma \ref{higher_dim_lemma1} that 
\[
\norm{\nabla u_1-\nabla u_0}_{L^p(\Omega)} \leq C (\norm{\sigma_1-\sigma_0}_{L^{\infty}(\Omega)} + \norm{A_1-A_0}_{L^{\infty}(\Omega)})^{\min\{ \frac{1}{p-1}, 1 \}}.
\]
On the other hand, \eqref{uj_conebeta_regularity} implies 
\[
\norm{\nabla u_1-\nabla u_0}_{C^{\beta}(\overline{\Omega})} \leq C.
\]
The lemma follows by interpolating the last two estimates by using Lemma \ref{inter} in the appendix.
\end{proof}

Now we show that the linear function $u_0(x) = x_1$ solving 
\[
\mdiv(\abs{\nabla u_0}^{p-2} \nabla u_0) = 0 \text{ in $\Omega$}
\]
can be perturbed into a solution of $\mdiv(\sigma \abs{A \nabla u \cdot \nabla u}^{\frac{p-2}{2}} A \nabla u) = 0$ having nonvanishing gradient, if $\sigma$ and $A$ are sufficiently close to constant.

\begin{lemma}\label{interiorHigh}
Let $\Omega \subset \mR^n$, $n \geq 2$, be a bounded open set with $C^{1,\alpha}$ boundary where $0 < \alpha < 1$, let $1 < p < \infty$, and let $M > 0$. There exists $\eps = \eps(n,p,\alpha,\Omega,M) > 0$ such that for any $\sigma \in C^{\alpha}(\overline{\Omega})$ and for any symmetric positive definite $A \in C^{\alpha}(\overline{\Omega}, \mR^{n \times n})$ satisfying 
\begin{gather*}
\norm{\sigma}_{C^{\alpha}(\overline{\Omega})} + \norm{A}_{C^{\alpha}(\overline{\Omega})} \leq M, \\
\norm{\sigma-1}_{L^{\infty}(\Omega)} + \norm{A-I}_{L^{\infty}(\Omega)} \leq \eps,
\end{gather*}
there exists a solution $u \in C^1(\overline{\Omega})$ of 
\[
\mdiv(\sigma \abs{A \nabla u \cdot \nabla u}^{\frac{p-2}{2}} A \nabla u) = 0 \text{ in $\Omega$}
\]
satisfying $\abs{\nabla u} \geq 1/2$ in $\Omega$.
\end{lemma}
\begin{proof}
Note that by taking $\eps$ small enough, one has 
\[
1/2 \leq \sigma \leq 2, \qquad 1/2 \leq A \leq 2 \ \text{ in $\Omega$}.
\]
Choose $\sigma_1 = \sigma$, $A_1 = A$ and $\sigma_0 = 1$, $A_0 = I$, and observe that the linear function $u_0(x) = x_1$ solves the $p$-Laplace equation 
\[
\mdiv(\sigma_0 \abs{A_0 \nabla u_0 \cdot \nabla u_0}^{\frac{p-2}{2}} A_0 \nabla u_0) = 0 \text{ in $\Omega$}.
\]
By Lemma \ref{higher_dim_lemma2}, there are $C > 0$ and $\gamma > 0$ so that the solution $u = u_1$ of 
\[
\mdiv(\sigma_1 \abs{A_1 \nabla u \cdot \nabla u}^{\frac{p-2}{2}} A_1 \nabla u) = 0 \text{ in $\Omega$}, \qquad u|_{\partial \Omega} = u_0|_{\partial \Omega}
\]
satisfies 
\[
\norm{\nabla u - \nabla u_0}_{L^{\infty}(\Omega)} \leq C (\norm{\sigma-1}_{L^{\infty}(\Omega)} + \norm{A-I}_{L^{\infty}(\Omega)})^{\gamma}.
\]
If $\eps$ is chosen so that $C(2\eps)^{\gamma} \leq 1/2$, we have 
\[
\abs{\nabla u} \geq \abs{\nabla u_0} - \abs{\nabla u - \nabla u_0} \geq 1/2 \text{ in $\Omega$.} \qedhere
\]
\end{proof}

\begin{proof}[\it{Proof of Theorem \ref{thm2}}]
First choose $\eps$ as in Lemma \ref{interiorHigh}, and choose $u \in W^{1,p}(\Omega)$ so that $u$ solves 
\[
\mdiv(\sigma_2 \abs{A \nabla u \cdot \nabla u}^{\frac{p-2}{2}} A \nabla u) = 0 \text{ in $\Omega$}
\]
and satisfies $\abs{\nabla u} \geq 1/2$ in $\Omega$. We now use a similar argument as in the proof of Theorem \ref{main_theorem}. The conditions $\sigma_1 \geq \sigma_2$ in $\Omega$ and $\Lambda_{\sigma_1} = \Lambda_{\sigma_2}$ together with the monotonicity inequality, Lemma \ref{needed_lemma}, imply that 
\[
\abs{\nabla u_2}^p = 0 \text{ a.e.\ in $E$}
\]
for any $u_2$ solving $\mdiv(\sigma_2 \abs{A\nabla u_2\cdot \nabla u_2}^{(p-2)/2} A\nabla u_2) = 0$, where $E = \{Êx \in \Omega \,;\, \sigma_1(x) > \sigma_2(x) \}$. If the open set $E$ were nonempty, one could choose $u_2 = u$ to obtain a contradiction. Thus $E$ must be empty and we have $\sigma_1 = \sigma_2$ in $\Omega$.
\end{proof}

%

\appendix

\section{Auxiliary results} \label{sec_appendix}

In this appendix, we first prove a result related to the decomposition of a positive definite matrix having continuous entries and then state an interpolation result. Finally we finish this section by recalling a proof of the unique continuation principle for the weighted $p$-Laplace equation in the plane.

\subsection{Matrix decomposition}
\begin{lemma}\label{existence_B}
Let $A \in C(\overline{\Omega}, \mR^{n \times n})$ be an $n\times n$ symmetric positive definite matrix function. Then there exists a matrix function $B \in C(\overline{\Omega}, \mR^{n \times n})$ such that $A=B^{\top}B$.
\end{lemma}
\begin{proof}
Consider the following inner product and norm on $\mR^n$ defined for $x \in \Omega$, 
\[
\langle v, w \rangle_{A(x)} = A(x) v \cdot w, \qquad \abs{v}_{A(x)} = (A(x)v \cdot v)^{1/2}, \qquad v, w \in \mR^n.
\]
We apply the Gram-Schmidt orthonormalization procedure to the standard basis $\{ e_1, \ldots, e_n \}$ of $\mR^n$ with respect to this inner product. Define 
\[
w_1(x) = e_1, \qquad v_1(x) = w_1/\abs{w_1}_{A(x)},
\]
and if $k \geq 2$ define inductively 
\begin{align*}
w_k(x) &= e_k - \langle e_k, v_1(x) \rangle_{A(x)} v_1(x) - \ldots - \langle e_k, v_{k-1}(x) \rangle_{A(x)} v_{k-1}(x), \\
v_k(x) &= w_k(x)/\abs{w_k(x)}_{A(x)}.
\end{align*}
Now $v_1(x) = e_1/\sqrt{a_{11}(x)}$ is a continuous vector function in $x$, and inductively one sees that each $v_k(x)$ is also continuous in $x$. Here it is crucial that $A(x)$ is positive definite and that $\{Êe_1, \ldots, e_n \}$ are linearly independent, so the denominators $\abs{w_k(x)}_{A(x)}$ are positive and continuous.

The above process leads to a basis $\{ v_1(x), \ldots, v_n(x) \}$ of $\mR^n$ which is orthonormal in the $A(x)$ inner product, i.e.\ $A(x) v_j(x) \cdot v_k(x) = \delta_{jk}$. Then $V(x) = \left( \begin{array}{ccc} v_1(x) & \ldots & v_n(x) \end{array} \right)$ is a matrix function in $C(\overline{\Omega}, \mR^{n \times n})$ and satisfies 
\[
V(x)^{\top} A(x) V(x) = I.
\]
Linear independence implies that $\det(V(x)) \neq 0$ for all $x \in \Omega$. It follows that the matrix $B(x) = V(x)^{-1}$ is in $C(\overline{\Omega}, \mR^{n \times n})$ and satisfies $B(x)^{\top} B(x) = A(x)$ in $\Omega$.
\end{proof}

\subsection{Interpolation}

\begin{lemma} \label{inter}
Let $\Omega$ be a bounded Lipschitz domain in $\mathbb{R}^n, n\geq 2,$ let $0<\beta< 1$, and let $1 < p < \infty$. For any $\theta \in (\frac{n/p}{\beta+n/p},1]$ there is $C > 0$ such that whenever $f\in C^{\beta}(\overline{\Omega})$ satisfies 
\begin{align*}
\|f\|_{L^p(\Omega)} &\leq M_0, \\
\|f\|_{C^{\beta}(\overline{\Omega})} &\leq M_1,
\end{align*}
then 
\[
\|f\|_{L^{\infty}(\Omega)} \leq C M_0^{1-\theta} M_1^{\theta}.
\]
\end{lemma}
\begin{proof}
Recall that, for $0<\beta < 1,$ the H\"older space $C^{\beta}(\overline{\Omega})$ is precisely the Besov space $B^{\beta}_{\infty\infty}(\Omega)$. We also denote by $W^{s,p}(\Omega)$ the $L^p$ Sobolev space with smoothness index $s$.

By the results in \cite[Section 23]{Triebel:1997}, whenever $\varepsilon > 0$ and $2 \leq q < \infty$ one has the continuous embeddings (which in fact hold for any bounded domain $\Omega$, not necessarily Lipschitz):
\[
B^{\beta}_{\infty\infty}(\Omega) \subset B^{\beta-\varepsilon}_{q2}(\Omega) \subset F^{\beta-\varepsilon}_{q2}(\Omega) = W^{\beta-\varepsilon,q}(\Omega).
\]
Thus we see that 
\begin{align*}
\norm{f}_{W^{0,p}(\Omega)} &\leq M_0, \\
\norm{f}_{W^{\beta-\varepsilon,q}(\Omega)} &\leq C M_1.
\end{align*}
By complex interpolation \cite[Theorem 2.13]{Triebel:2002}, we obtain for any $0 \leq t \leq 1$ that 
\[
\norm{f}_{W^{s_t,r_t}(\Omega)} \leq M_0^{1-t} (C M_1)^{t}
\]
where $s_t = t(\beta-\varepsilon)$ and $\frac{1}{r_t} = (1-t) \frac{1}{p} + t \frac{1}{q}$.

Now fix $\theta$ with $\frac{n/p}{\beta+n/p} < \theta \leq 1$, and fix $\varepsilon > 0$ and $2 \leq q < \infty$ so that $\theta(\beta - \varepsilon) > \frac{n}{r}$ where $\frac{1}{r} = (1-\theta) \frac{1}{p} + \theta \frac{1}{q}$ (this condition is equivalent with $\theta(\beta - \varepsilon + \frac{n}{p} - \frac{n}{q}) > \frac{n}{p}$, and such $\varepsilon$ and $q$ exist since $\theta > \frac{n/p}{\beta+n/p}$). Choosing $t = \theta$ above and using the continuous embedding $W^{s_{\theta}, r_{\theta}}(\Omega) \subset L^{\infty}(\Omega)$, which follows from \cite[Section 23]{Triebel:1997} since $s_{\theta} > n/r_{\theta}$, we get  
\[
\norm{f}_{L^{\infty}(\Omega)} \leq C M_0^{1-\theta} (C M_1)^{\theta}
\]
as required. (After the initial embeddings, one could also use the universal extension operator for Lipschitz domains \cite[Theorem 2.11]{Triebel:2002} and work in $\mathbb{R}^n$.)
\end{proof}

\subsection{Unique continuation principle for weighted $p$-harmonic functions in the plane}
In this subsection, we sketch a proof of the unique continuation principle for the weighted $p$-Laplace equation
\begin{equation}\label{p-lap}
\mdiv(\sigma \abs{\nabla u}^{p-2} \nabla u) = 0
\end{equation}
in a bounded domain $\Omega \subset \mathbb{R}^2$ following \cite{bi84}. This result is also a special case of \cite[Proposition 3.3]{Alessandrini:Sigalotti:2001}. Indeed, according to a very recent work of \cite{Guo:Kar:2015}, even the strong unique continuation principle holds for the solutions of~\eqref{p-lap}. Assume that $\sigma$ is positive and Lipschitz continuous in $\Omega$. 

We first consider the case $p\geq 2$. Let us define a vector field $F : \Omega \rightarrow \mathbb{R}^2$ by 
$$F(x) = \sigma^{p/2}|\nabla u(x)|^{(p-2)/2}\nabla u(x),$$
where $u$ satisfies~\eqref{p-lap}. Then it follows from~\cite[Proof of Proposition 2]{bi84} that $F\in W^{1,2}_{loc}(\Omega, \R^2)$.

We write $f = \sigma u_x - i\sigma u_y$ for the complex gradient of $u$ in the complex plane, where $z = x+iy$, and define the nonlinear counterparts of the complex gradient $f$ by $F_a = |f|^a f, a>-1$. In the following computations, we will derive the nonlinear first order elliptic system for $F$, that is $F_a$ with $a = (p-2)/2$. 
		
From the definition of $f$ and $F_a$, we have
		\[
		2u_x = \frac{1}{\sigma}|F_a|^{-\frac{a}{a+1}}\big(F_a + \overline{F_a}\big)
		\]
		and 
		\[
		2u_y = i \frac{1}{\sigma}|F_a|^{-\frac{a}{a+1}}\big(F_a - \overline{F_a}\big).
		\]
Therefore from the above equalities we have,
		\[
		\frac{\partial}{\partial y}\Big[\frac{1}{\sigma}|F_a|^{-\frac{a}{a+1}}(F_a + \overline{F_a})\Big]
		= i \frac{\partial}{\partial x}\Big[\frac{1}{\sigma}|F_a|^{-\frac{a}{a+1}}(F_a - \overline{F_a})\Big].
		\]
This is equivalent to
		\begin{equation}\label{ImcO}
		\text{Im} \frac{\partial}{\partial\bar{z}}\Big(\frac{1}{\sigma}|F_a|^{-\frac{a}{a+1}}F_a\Big) = 0. 
		\end{equation}
Note that, for $a={(p-2)}/2$, $F_a$ is differentiable almost everywhere and so \eqref{ImcO} reduces to the complex equation
		\begin{equation}\label{comEq}
		\begin{split}
		\frac{\partial}{\partial\bar{z}}F_a - \overline{\frac{\partial}{\partial\bar{z}}F_a}
		= &-\frac{a}{a+2}\left[\frac{\bar{F}_a}{F_a} \frac{\partial}{\partial z} F_a - \frac{F_a}{\bar{F}_a} \overline{\frac{\partial}{\partial z} F_a}\right] \\
		& + \sigma \frac{2a+2}{a+2}\left[\bar{F}_a \frac{\partial}{\partial z}\left(\frac{1}{\sigma}\right) - F_a \frac{\partial}{\partial \bar{z}}\left(\frac{1}{\sigma}\right)\right]. 
		\end{split}
		\end{equation}
On the other hand, since $u$ satisfies weighted $p$-Laplacian equation $\nabla\cdot(\sigma |\nabla u|^{p-2}\nabla u ) = 0$, we have
		\[
		\nabla\cdot\left[\frac{1}{\sigma^{p-2}}|F_a|^{\frac{p-2-a}{a+1}}(F_a+\overline{F_a}, i(F_a-\overline{F_a})) \right] = 0,
		\]
which is equivalent to the equation
		\[
		\frac{\partial}{\partial x}\left\{\frac{1}{\sigma^{p-2}}|F_a|^{\frac{p-2-a}{a+1}} (F_a+\overline{F_a}) \right\}
		+ i \frac{\partial}{\partial y}\left\{\frac{1}{\sigma^{p-2}}|F_a|^{\frac{p-2-a}{a+1}} (F_a-\overline{F_a}) \right\} = 0.
		\]
Using the complex notation we can write
		\[
		\text{Re} \frac{\partial}{\partial\bar{z}}\left(\frac{1}{\sigma^{p-2}}|F_a|^{\frac{p-2-a}{a+1}}F_a \right) = 0.
		\]
For $a=(p-2)/2$, $F_a$ is differentiable almost everywhere and so the last equation can be written as
		\begin{equation}\label{comEq2}
		\begin{split}
		\frac{\partial}{\partial\bar{z}}F_a + \overline{\frac{\partial}{\partial\bar{z}}F_a} 
		=& -\frac{p-2-a}{a+p}\left[\frac{\overline{F_a}}{F_a} \frac{\partial}{\partial z} F_a + \frac{F_a}{\overline{F_a}} \overline{\frac{\partial}{\partial z} F_a}\right] \\
		& - \sigma^{p-2} \frac{2a+2}{a+p}\left[F_a\frac{\partial}{\partial \bar{z}}\left(\frac{1}{\sigma^{p-2}}\right) + \overline{F_a} \frac{\partial}{\partial z}\left(\frac{1}{\sigma^{p-2}}\right)\right]. 
		\end{split}
		\end{equation}
Adding \eqref{comEq} and \eqref{comEq2}, we get (with $a=(p-2)/2$) 
		\begin{align}\label{eq:for Fa}
		\frac{\partial}{\partial \bar{z}}F=q_1\frac{\partial}{\partial z}F+q_2\overline{\frac{\partial}{\partial z}F}+H(z,F),
		\end{align}
where
		\[
		q_1 = -\frac{1}{2}\left(\frac{p-2}{p+2} + \frac{p-2}{3p-2} \right)\frac{\overline{F}}{F},
		\]
		\[
		q_2 = -\frac{1}{2}\left(\frac{p-2}{3p-2} - \frac{p-2}{p+2} \right)\frac{F}{\overline{F}}
		\]
and
		\[
		\begin{split}
		H(z,F) 
		=&  \sigma \frac{p}{p+2}\left[\overline{F} \frac{\partial}{\partial z}\left(\frac{1}{\sigma}\right) - F \frac{\partial}{\partial \bar{z}}\left(\frac{1}{\sigma}\right)\right] \\
		&- \sigma^{p-2} \frac{p}{3p-2}\left[ \overline{F} \frac{\partial}{\partial z}\left(\frac{1}{\sigma^{p-2}}\right) + F\frac{\partial}{\partial \bar{z}}\left(\frac{1}{\sigma^{p-2}}\right)\right].
		\end{split}
		\]
 It is easy to check that $|q_1|+|q_2|<1$ and $|H(z,F)|\leq q_3(z)|F|$ with $q_3\in L^\infty$. Under these structure assumptions for $q_1$, $q_2$ and $q_3$, by~\cite[Section 8.4]{b09}, the solution of~\eqref{eq:for Fa} can be represented as
\begin{align}\label{eq:representation formula}
F(z)=H(\chi(z))e^{\varphi(z)},
\end{align}
where $H$ is analytic, $\chi$ is quasiconformal and $\varphi$ is H\"older continuous in $\overline{\Omega}$ with $\varphi_{\bar{z}}, \varphi_{z}\in L^q(\Omega)$ for some $q>2$. Write $R(\xi)=|\xi|^{\frac{2a+2-p}{2p}}\xi$ and it is clear that 
\begin{align}\label{eq:general Fa rep}
F_a(z)=R\circ F_{(p-2)/2}=\big|H(\chi(z))\big|^\beta H(\chi(z))e^{(\beta+1)\varphi(z)}.
\end{align}
Note that the function $G_\beta:=\big|H(\chi(z))\big|^\beta H(\chi(z))$ is quasiregular as being the composition of a quasiregular mapping with an analytic function. 

Suppose now $u$ is constant on an open subset of $\Omega$, then $F_a$ will vanish on that open subset, which together with the observation that $e^{(\beta+1)\varphi(z)}$ is non-zero, implies that $G_\beta$ is zero on that open subset. Since $G_\beta$ is quasiregular, it is either constant or both discrete and open, and $G_\beta$ being zero in an open subset of $\Omega$ necessarily forces $G_\beta$ to be zero everywhere in $\Omega$. This implies that $F_a\equiv 0$ in $\Omega$ and hence $u$ is identically constant in $\Omega$.

The case $p\in (1,2)$ can be treated identically as above, provided that we are able to show $F\in W^{1,2}_{loc}(\Omega,\R^2)$. Note that since the regularity is a local issue, we may assume that $\Omega$ is a simply connected bounded domain. As in~\cite[Page 426-427]{aim09}, this can be done by a very elegant argument involving
the weighted dual $q$-harmonic equation, where $\frac{1}{p} + \frac{1}{q} = 1$. Namely, there exists a weighted $q$-harmonic function $v\in W^{1,q}_{loc}(\Omega)$ satisfying $$\dive({\sigma}^{1-q}|\nabla v|^{q-2}\nabla v) =0$$ such that
\[
v_x = -\sigma|\nabla u|^{p-2}u_y \ \text{and}\ v_y = \sigma|\nabla u|^{p-2}u_x.
\] 
Since $q>2$ and $\sigma^{1-q}$ is positive and Lipschitz, it follows again from \cite[Proof of Proposition 2]{bi84} that $|\nabla v|^{(q-2)/2}\nabla v\in W^{1,2}_{loc}(\Omega,\R^2)$ and so $F\in W^{1,2}_{loc}(\Omega,\R^2)$ as desired. 

\bibliographystyle{alpha}
\bibliography{math}

\end{document}